\documentclass[11pt,a4paper]{amsart}
\usepackage{amsfonts}
\usepackage{eurosym}
\usepackage{graphicx}
\usepackage{color}
\usepackage[colorlinks=true]{hyperref}
\usepackage{subfigure}

\hypersetup{urlcolor=blue, citecolor=blue, draft=false}

\newtheorem{theorem}{Theorem}[section]
\newtheorem{lemma}{Lemma}[section]
\newtheorem{corollary}{Corollary}[section]
\theoremstyle{remark}
\newtheorem{example}{Example}[section]

\theoremstyle{remark}

\input{tcilatex}

\begin{document}
\title[Partial sums of Normalized Dini Functions]{Partial sums of the
Normalized Dini Functions}
\author{HAL\.{I}T ORHAN$^{1}$, \.{I}BRAH\.{I}M AKTA\c{S}$^{2}$}
\address{$^{1}$Department of Mathematics, Faculty of Science, Atat\"{u}rk
University, Erzurum, 25240, Turkey.}
\email{orhanhalit607@gmail.com}
\address{$^{2}$Corresponding author, Department of Mathematical Engineering,
Faculty of Engineering and Natural Science, G\"{u}m\"{u}\c{s}hane
University, G\"{u}m\"{u}\c{s}hane, 29100, Turkey.}
\email{aktasibrahim38@gmail.com}
\date{17.05.2016}

\begin{abstract}
Let $\left( w_{\alpha ,v}\right)
_{m}(z)=z+\dsum\limits_{n=1}^{m}a_{n}z^{n+1} $ be the sequence of partial
sums of normalized Dini functions $w_{\alpha
,v}(z)=z+\dsum\limits_{n=1}^{\infty }a_{n}z^{n+1}$ where $a_{n}=\frac{\left(
-1\right) ^{n}\left( 2n+\alpha \right) }{\alpha 4^{n}n!\left( v+1\right) _{n}%
}$. The aim of the present paper is to obtain lower bounds for $\mathcal{R}%
\left\{ \frac{w_{\alpha ,v}(z)}{\left( w_{\alpha ,v}\right) _{m}(z)}\right\}
,$ $\mathcal{R}\left\{ \frac{\left( w_{\alpha ,v}\right) _{m}(z)}{w_{\alpha
,v}(z)}\right\} ,$ $\mathcal{R}\left\{ \frac{w_{\alpha ,v}^{^{\prime }}(z)}{%
\left( w_{\alpha ,v}\right) _{m}^{^{\prime }}(z)}\right\} $ and $\mathcal{R}%
\left\{ \frac{\left( w_{\alpha ,v}\right) _{m}^{^{\prime }}(z)}{w_{\alpha
,v}^{^{\prime }}(z)}\right\} $. Also we give a few geometric description
regarding image domains of some functions.
\end{abstract}

\maketitle

\section{Introduction}

Let $\mathcal{A}$ denote the family of functions $f$ of the form 
\begin{equation}
f(z)=z+\sum\limits_{k=1}^{\infty }a{_{k}z^{k+1}}  \label{1.1}
\end{equation}%
which are analytic in the open unit disk $\mathcal{U}=\left\{ z\in \mathcal{C%
}:\left\vert z\right\vert <1\right\} $ and satisfy the usual normalization
condition $f(0)=f^{\prime }(0)-1=0.$ Let $\mathcal{S}$ denote the subclass
of $\mathcal{A}$ which are univalent in $\mathcal{U}$. Also let $\mathcal{S}%
^{\ast }$and $\mathcal{C}$ denote the subclasses of $\mathcal{A}$ consisting
of functions which are starlike and convex in $\mathcal{U}$,respectively.
These classes can characterize as follows:%
\begin{equation}
f\in \mathcal{S}^{\ast }\Leftrightarrow \Re \left( \frac{zf^{\prime }(z)}{%
f(z)}\right) >0,  \label{1.2}
\end{equation}%
\begin{equation}
f\in \mathcal{C}\Leftrightarrow \Re \left( 1+\frac{zf^{\prime \prime }(z)}{%
f^{\prime }(z)}\right) >0.  \label{1.3}
\end{equation}%
\qquad\ \ \ \ \ 

Some special functions like Bessel functions play an important role in
applied mathematics and physics. We know that the Bessel functions of the
first kind $J_{v}$ is defined by(see \cite{2} and \cite{12})%
\begin{equation}
J_{v}(z)=\dsum\limits_{n=0}^{\infty }\frac{\left( -1\right) ^{n}}{n!\Gamma
\left( v+n+1\right) }\left( \frac{z}{2}\right) ^{2n+v}\text{ }(z\in \mathcal{%
C})  \label{1.4}
\end{equation}%
where $\Gamma (z)$ stands for Euler gamma function and it is a particular
solution of the second-order linear homogeneous differential equation(see,
for details, \cite{2} and \cite{12}):%
\begin{equation}
z^{2}y^{^{\prime \prime }}(z)+zy^{^{\prime }}(z)+\left[ z^{2}-v^{2}\right]
y(z)=0  \label{1.5}
\end{equation}%
where $v\in \mathcal{C}.$

Now, we consider the function $w_{\alpha ,v}:\mathcal{U}\rightarrow \mathcal{%
C},$ defined by(see \cite{4}) 
\begin{eqnarray}
w_{\alpha ,v}(z) &=&\frac{2^{v}}{\alpha }\Gamma (v+1)z^{1-\frac{v}{2}%
}((\alpha -v)J_{v}(\sqrt{z})+\sqrt{z}J_{v}^{^{\prime }}(\sqrt{z}))
\label{1.6} \\
&=&\dsum\limits_{n=0}^{\infty }\frac{\left( -1\right) ^{n}(2n+\alpha )\Gamma
(v+1)z^{n+1}}{\alpha .4^{n}n!\Gamma \left( v+n+1\right) }.  \notag
\end{eqnarray}%
On the other hand we know that there is following relation in between
Pochhammer Symbol (or Appell) and Euler's gamma function:%
\begin{equation}
\left( \mu \right) _{n}=\frac{\Gamma (\mu +n)}{\Gamma (\mu )}=\mu (\mu
+1)...(\mu +n-1).  \label{1.7}
\end{equation}%
If we use the relation (1.7) we can obtain the following series
representation for the functions $w_{\alpha ,v}(z)$ given by (1.6):%
\begin{equation}
w_{\alpha ,v}(z)=z+\dsum\limits_{n=1}^{\infty }a_{n}z^{n+1}  \label{1.8}
\end{equation}%
where $a_{n}=\frac{\left( -1\right) ^{n}(2n+\alpha )}{\alpha
.4^{n}n!(v+1)_{n}}.$

In this study, we will investigate the ratio of a function of the form (1.8)
to its sequence of partial sums $\left( w_{\alpha ,v}(z)\right)
_{m}=z+\dsum\limits_{n=1}^{m}a_{n}z^{n+1}$ when the coefficients of $%
w_{\alpha ,v}$ satisfy some conditions. We will obtain lower bounds for $%
\mathcal{R}\left\{ \frac{w_{\alpha ,v}(z)}{\left( w_{\alpha ,v}\right)
_{m}(z)}\right\} ,$ $\mathcal{R}\left\{ \frac{\left( w_{\alpha ,v}\right)
_{m}(z)}{w_{\alpha ,v}(z)}\right\} ,$ $\mathcal{R}\left\{ \frac{w_{\alpha
,v}^{^{\prime }}(z)}{\left( w_{\alpha ,v}\right) _{m}^{^{\prime }}(z)}%
\right\} $ and $\mathcal{R}\left\{ \frac{\left( w_{\alpha ,v}\right)
_{m}^{^{\prime }}(z)}{w_{\alpha ,v}^{^{\prime }}(z)}\right\} .$

For interesting developments on the partial sums of some special functions
and the some classes of analytic functions, the readers can browse to the
works of Orhan and Yagmur \cite{7}, \c{C}a\u{g}lar and Deniz \cite{6},
Brickman et al. \cite{5}, Silverman \cite{10}, Owa et al. \cite{8},
Sheil-Small \cite{9}, Silvia \cite{11}, Akta\c{s} and Orhan \cite{1}.

\begin{lemma}
If $\alpha >0$ and $v>-\frac{7}{8},$ then the function%
\begin{equation*}
w_{\alpha ,v}:\mathcal{U}\rightarrow \mathcal{C}
\end{equation*}%
given by (1.6) satisfies the follwing inequalities:%
\begin{equation}
\left\vert w_{\alpha ,v}(z)\right\vert \leq 1+\frac{32v+2\alpha +32}{\left(
8v+7\right) \alpha },  \label{1.9}
\end{equation}%
\begin{equation}
\left\vert w_{\alpha ,v}^{^{\prime }}(z)\right\vert \leq 1+\frac{256\left(
2+\alpha \right) v^{2}+8\left( 64+29\alpha \right) v+210\alpha +512}{\left(
8v+7\right) ^{3}\alpha }.  \label{1.10}
\end{equation}
\end{lemma}

\begin{proof}
\textit{By using the well-known triangle inequality:}%
\begin{equation}
\left\vert z_{1}+z_{2}\right\vert \leq \left\vert z_{1}\right\vert
+\left\vert z_{2}\right\vert  \label{1.11}
\end{equation}%
\textit{\ and the inequalities }

\begin{equation}
\left( v+1\right) ^{n}\leq \left( v+1\right) _{n}\text{ and }2^{n-1}\leq n!
\label{1.12}
\end{equation}%
for $n\in \mathcal{%
\mathbb{N}
}$\textit{\ }$=\left\{ 1,2,...\right\} ,$ we get%
\begin{eqnarray*}
\left\vert w_{\alpha ,v}(z)\right\vert &=&\left\vert
z+\dsum\limits_{n=1}^{\infty }\frac{\left( -1\right) ^{n}(2n+\alpha )}{%
\alpha .4^{n}n!(v+1)_{n}}z^{n+1}\right\vert \leq
1+\dsum\limits_{n=1}^{\infty }\frac{2n+\alpha }{\alpha .4^{n}2^{n-1}(v+1)^{n}%
} \\
&=&1+\frac{2}{4\alpha \left( v+1\right) }\dsum\limits_{n=1}^{\infty }\frac{n%
}{\left( 8\left( v+1\right) \right) ^{n-1}}+\frac{1}{4\left( v+1\right) }%
\dsum\limits_{n=1}^{\infty }\left( \frac{1}{8\left( v+1\right) }\right)
^{n-1} \\
&=&1+\frac{32v+2\alpha +32}{\left( 8v+7\right) \alpha }.
\end{eqnarray*}%
\textit{If we consider inequality (1.10), then we have}%
\begin{eqnarray*}
\left\vert w_{\alpha ,v}^{^{\prime }}(z)\right\vert &=&\left\vert
1+\dsum\limits_{n=1}^{\infty }\frac{\left( -1\right) ^{n}(2n+\alpha )\left(
n+1\right) }{\alpha .4^{n}n!(v+1)_{n}}z^{n}\right\vert \leq
1+\dsum\limits_{n=1}^{\infty }\frac{2n^{2}+\left( 2+\alpha \right) n+\alpha 
}{\alpha .4^{n}2^{n-1}(v+1)^{n}} \\
&=&1+\frac{1}{\left( v+1\right) }\left\{ \frac{1}{2\alpha }%
\dsum\limits_{n=1}^{\infty }\frac{n^{2}}{\left( 8\left( v+1\right) \right)
^{n-1}}+\frac{\left( 2+\alpha \right) }{4\alpha }\dsum\limits_{n=1}^{\infty }%
\frac{n}{\left( 8\left( v+1\right) \right) ^{n-1}}+\frac{1}{4}%
\dsum\limits_{n=1}^{\infty }\left( \frac{1}{8\left( v+1\right) }\right)
^{n-1}\right\} \\
&=&1+\frac{256\left( 2+\alpha \right) v^{2}+8\left( 64+29\alpha \right)
v+210\alpha +512}{\left( 8v+7\right) ^{3}\alpha }.
\end{eqnarray*}
\end{proof}

\section{Main Results}

\begin{theorem}
Let $\alpha >0,$ $v>-\frac{7}{8},$ the function $w_{\alpha ,v}:\mathcal{U}%
\rightarrow \mathcal{C}$ defined by (1.8) and its sequence of partial sums $%
\left( w_{\alpha ,v}(z)\right) _{m}=z+\dsum\limits_{n=1}^{m}a_{n}z^{n+1}$.
If the inequality%
\begin{equation*}
8\left( \alpha -4\right) v+5\alpha -32\geq 0
\end{equation*}%
is true, then we have the followings:%
\begin{equation}
\mathcal{R}\left\{ \frac{w_{\alpha ,v}(z)}{\left( w_{\alpha ,v}\right)
_{m}(z)}\right\} \geq \frac{8\left( \alpha -4\right) v+5\alpha -32}{\left(
8v+7\right) \alpha },  \tag{2.1}
\end{equation}%
\begin{equation}
\mathcal{R}\left\{ \frac{\left( w_{\alpha ,v}\right) _{m}(z)}{w_{\alpha
,v}(z)}\right\} \geq \frac{\left( 8v+7\right) \alpha }{\left( 8v+7\right)
\alpha +32v+2\alpha +32}.  \tag{2.2}
\end{equation}
\end{theorem}

\begin{proof}
From the inequality (1.9) in the Lemma 1.1 we write that%
\begin{equation*}
1+\dsum\limits_{n=1}^{\infty }\left\vert a_{n}\right\vert \leq 1+\frac{%
32v+2\alpha +32}{\left( 8v+7\right) \alpha }
\end{equation*}%
which is equivalent to%
\begin{equation*}
\frac{\left( 8v+7\right) \alpha }{32v+2\alpha +32}\dsum\limits_{n=1}^{\infty
}\left\vert a_{n}\right\vert \leq 1,
\end{equation*}%
where $a_{n}=\frac{\left( -1\right) ^{n}(2n+\alpha )}{\alpha
.4^{n}n!(v+1)_{n}}.$

Now we may write%
\begin{equation}
\frac{\left( 8v+7\right) \alpha }{32v+2\alpha +32}\left\{ \frac{w_{\alpha
,v}(z)}{\left( w_{\alpha ,v}\right) _{m}(z)}-\left( 1-\frac{32v+2\alpha +32}{%
\left( 8v+7\right) \alpha }\right) \right\}  \tag{2.3}
\end{equation}%
\begin{eqnarray*}
&=&\frac{1+\dsum\limits_{n=1}^{m}a_{n}z^{n}+\frac{\left( 8v+7\right) \alpha 
}{32v+2\alpha +32}\dsum\limits_{n=m+1}^{\infty }a_{n}z^{n}}{%
1+\dsum\limits_{n=1}^{m}a_{n}z^{n}} \\
&=&\frac{1+A(z)}{1+B(z)}.
\end{eqnarray*}

Set $\left( 1+A(z)\right) $/$\left( 1+B(z)\right) =(1+p(z))/(1-p(z)),$ so
that $p(z)=\left( A(z)-B(z)\right) /\left( 2+A(z)+B(z)\right) .$ Then%
\begin{equation*}
p(z)=\frac{\frac{\left( 8v+7\right) \alpha }{32v+2\alpha +32}%
\dsum\limits_{n=m+1}^{\infty }a_{n}z^{n}}{2+2\dsum%
\limits_{n=1}^{m}a_{n}z^{n}+\frac{\left( 8v+7\right) \alpha }{32v+2\alpha +32%
}\dsum\limits_{n=m+1}^{\infty }a_{n}z^{n}}
\end{equation*}%
and 
\begin{equation*}
\left\vert p(z)\right\vert \leq \frac{\frac{\left( 8v+7\right) \alpha }{%
32v+2\alpha +32}\dsum\limits_{n=m+1}^{\infty }\left\vert a_{n}\right\vert }{%
2-2\dsum\limits_{n=1}^{m}\left\vert a_{n}\right\vert -\frac{\left(
8v+7\right) \alpha }{32v+2\alpha +32}\dsum\limits_{n=m+1}^{\infty
}\left\vert a_{n}\right\vert }.
\end{equation*}%
Now $\left\vert p(z)\right\vert \leq 1$ if and only if%
\begin{equation}
\dsum\limits_{n=1}^{m}\left\vert a_{n}\right\vert +\frac{\left( 8v+7\right)
\alpha }{32v+2\alpha +32}\dsum\limits_{n=m+1}^{\infty }\left\vert
a_{n}\right\vert \leq 1.  \tag{2.4}
\end{equation}%
It suffices to show that the left hand side of (2.4) is bounded above by $%
\frac{\left( 8v+7\right) \alpha }{32v+2\alpha +32}\dsum\limits_{n=1}^{\infty
}\left\vert a_{n}\right\vert ,$ which is equivalent to%
\begin{equation*}
\frac{8\left( \alpha -4\right) v+5\alpha -32}{32v+2\alpha +32}%
\dsum\limits_{n=1}^{m}\left\vert a_{n}\right\vert \geq 0.
\end{equation*}%
To prove the result (2.2), we write%
\begin{equation}
\left( \frac{\left( 8v+7\right) \alpha }{32v+2\alpha +32}+1\right) \left\{ 
\frac{\left( w_{\alpha ,v}\right) _{m}(z)}{w_{\alpha ,v}(z)}-\frac{\left(
8v+7\right) \alpha }{\left( 8v+7\right) \alpha +32v+2\alpha +32}\right\} 
\tag{2.5}
\end{equation}%
\begin{eqnarray*}
&=&\frac{1+\dsum\limits_{n=1}^{m}a_{n}z^{n}-\frac{\left( 8v+7\right) \alpha 
}{32v+2\alpha +32}\dsum\limits_{n=m+1}^{\infty }a_{n}z^{n}}{%
1+\dsum\limits_{n=1}^{\infty }a_{n}z^{n}} \\
&=&\frac{1+p(z)}{1-p(z)}.
\end{eqnarray*}

Then from (2.5) we get%
\begin{equation*}
p(z)=\frac{-\frac{8(\alpha +4)v+9\alpha +32}{32v+2\alpha +32}%
\dsum\limits_{n=m+1}^{\infty }a_{n}z^{n}}{2+2\dsum%
\limits_{n=1}^{m}a_{n}z^{n}-\frac{8(\alpha -4)v+5\alpha -32}{32v+2\alpha +32}%
\dsum\limits_{n=m+1}^{\infty }a_{n}z^{n}}
\end{equation*}%
and%
\begin{equation*}
\left\vert p(z)\right\vert \leq \frac{\frac{8(\alpha +4)v+9\alpha +32}{%
32v+2\alpha +32}\dsum\limits_{n=m+1}^{\infty }\left\vert a_{n}\right\vert }{%
2-2\dsum\limits_{n=1}^{m}\left\vert a_{n}\right\vert -\frac{8(\alpha
-4)v+5\alpha -32}{32v+2\alpha +32}\dsum\limits_{n=m+1}^{\infty }\left\vert
a_{n}\right\vert }.
\end{equation*}%
Now $\left\vert p(z)\right\vert \leq 1$ if and only if%
\begin{equation}
\dsum\limits_{n=1}^{m}\left\vert a_{n}\right\vert +\frac{\left( 8v+7\right)
\alpha }{32v+2\alpha +32}\dsum\limits_{n=m+1}^{\infty }\left\vert
a_{n}\right\vert \leq 1.  \tag{2.6}
\end{equation}%
Since the left hand side of (2.6) is bounded above by $\frac{\left(
8v+7\right) \alpha }{32v+2\alpha +32}\dsum\limits_{n=1}^{\infty }\left\vert
a_{n}\right\vert ,$ the proof is complated.
\end{proof}

\begin{theorem}
Let $\alpha >0,$ $v>-\frac{7}{8},$ the function $w_{\alpha ,v}:\mathcal{U}%
\rightarrow \mathcal{C}$ defined by (1.8) and its sequence of partial sums $%
\left( w_{\alpha ,v}(z)\right) _{m}=z+\dsum\limits_{n=1}^{m}a_{n}z^{n+1}$.
If the inequality%
\begin{equation*}
512\alpha v^{3}+64\left( 17\alpha -8\right) v^{2}+16\left( 59\alpha
-32\right) v+133\alpha -512\geq 0
\end{equation*}%
is true, then we have the followings:%
\begin{equation}
\mathcal{R}\left\{ \frac{w_{\alpha ,v}^{^{\prime }}(z)}{\left( w_{\alpha
,v}\right) _{m}^{^{\prime }}(z)}\right\} \geq \frac{512\alpha v^{3}+64\left(
17\alpha -8\right) v^{2}+16\left( 59\alpha -32\right) v+133\alpha -512}{%
256\left( 2+\alpha \right) v^{2}+8\left( 64+29\alpha \right) v+210\alpha +512%
},  \tag{2.7}
\end{equation}%
\begin{equation}
\mathcal{R}\left\{ \frac{\left( w_{\alpha ,v}\right) _{m}^{^{\prime }}(z)}{%
w_{\alpha ,v}^{^{\prime }}(z)}\right\} \geq \frac{\left( 8v+7\right)
^{3}\alpha }{512\alpha v^{3}+64\left( 25\alpha +8\right) v^{2}+128\left(
11\alpha +4\right) v+553\alpha +512}.  \tag{2.8}
\end{equation}
\end{theorem}

\begin{proof}
From the inequality (1.10) in the Lemma 1.1 we write that%
\begin{equation*}
1+\dsum\limits_{n=1}^{\infty }\left( n+1\right) \left\vert a_{n}\right\vert
\leq 1+\frac{256\left( 2+\alpha \right) v^{2}+8\left( 64+29\alpha \right)
v+210\alpha +512}{\left( 8v+7\right) ^{3}\alpha }
\end{equation*}%
which is equivalent to 
\begin{equation*}
\frac{\left( 8v+7\right) ^{3}\alpha }{256\left( 2+\alpha \right)
v^{2}+8\left( 64+29\alpha \right) v+210\alpha +512}\dsum\limits_{n=1}^{%
\infty }\left( n+1\right) \left\vert a_{n}\right\vert \leq 1
\end{equation*}%
where $a_{n}=\frac{\left( -1\right) ^{n}(2n+\alpha )}{\alpha
.4^{n}n!(v+1)_{n}}.$

Now we write%
\begin{eqnarray*}
&&\frac{\left( 8v+7\right) ^{3}\alpha }{256\left( 2+\alpha \right)
v^{2}+8\left( 64+29\alpha \right) v+210\alpha +512}\times \\
&&\left\{ \frac{w_{\alpha ,v}^{^{\prime }}(z)}{\left( w_{\alpha ,v}\right)
_{m}^{^{\prime }}(z)}-(1-\frac{256\left( 2+\alpha \right) v^{2}+8\left(
64+29\alpha \right) v+210\alpha +512}{\left( 8v+7\right) ^{3}\alpha }%
)\right\} \\
&=&\frac{1+\dsum\limits_{n=1}^{m}\left( n+1\right) a_{n}z^{n}+\frac{\left(
8v+7\right) ^{3}\alpha }{256\left( 2+\alpha \right) v^{2}+8\left(
64+29\alpha \right) v+210\alpha +512}\dsum\limits_{n=m+1}^{\infty }\left(
n+1\right) a_{n}z^{n}}{1+\dsum\limits_{n=1}^{m}\left( n+1\right) a_{n}z^{n}}
\\
&=&\frac{1+p(z)}{1-p(z)}
\end{eqnarray*}%
where%
\begin{equation*}
\left\vert p(z)\right\vert \leq \frac{\frac{\left( 8v+7\right) ^{3}\alpha }{%
256\left( 2+\alpha \right) v^{2}+8\left( 64+29\alpha \right) v+210\alpha +512%
}\dsum\limits_{n=m+1}^{\infty }\left( n+1\right) \left\vert a_{n}\right\vert 
}{2-2\dsum\limits_{n=1}^{m}\left( n+1\right) \left\vert a_{n}\right\vert -%
\frac{\left( 8v+7\right) ^{3}\alpha }{256\left( 2+\alpha \right)
v^{2}+8\left( 64+29\alpha \right) v+210\alpha +512}\dsum\limits_{n=m+1}^{%
\infty }\left( n+1\right) \left\vert a_{n}\right\vert }\leq 1.
\end{equation*}%
The last inequality is equivalent to%
\begin{equation}
\dsum\limits_{n=1}^{m}\left( n+1\right) \left\vert a_{n}\right\vert +\frac{%
\left( 8v+7\right) ^{3}\alpha }{256\left( 2+\alpha \right) v^{2}+8\left(
64+29\alpha \right) v+210\alpha +512}\dsum\limits_{n=m+1}^{\infty }\left(
n+1\right) \left\vert a_{n}\right\vert \leq 1.  \tag{2.9}
\end{equation}%
It suffices to show that the left hand side of (2.9) is bounded above by

$\frac{\left( 8v+7\right) ^{3}\alpha }{256\left( 2+\alpha \right)
v^{2}+8\left( 64+29\alpha \right) v+210\alpha +512}\dsum\limits_{n=1}^{%
\infty }\left( n+1\right) \left\vert a_{n}\right\vert ,$ \ which is
equivalent to 
\begin{equation*}
\frac{512\alpha v^{3}+64\left( 17\alpha -8\right) v^{2}+16\left( 59\alpha
-32\right) +133\alpha -512}{256\left( 2+\alpha \right) v^{2}+8\left(
64+29\alpha \right) v+210\alpha +512}\dsum\limits_{n=1}^{m}\left( n+1\right)
\left\vert a_{n}\right\vert \geq 0.
\end{equation*}%
To prove the result (2.8), we write 
\begin{eqnarray*}
\frac{1+p(z)}{1-p(z)} &=&\left( \frac{\left( 8v+7\right) ^{3}\alpha }{%
256\left( 2+\alpha \right) v^{2}+8\left( 64+29\alpha \right) v+210\alpha +512%
}+1\right) \times \\
&&\left\{ \frac{\left( w_{\alpha ,v}\right) _{m}^{^{\prime }}(z)}{w_{\alpha
,v}^{^{\prime }}(z)}-\frac{\left( 8v+7\right) ^{3}\alpha }{512\alpha
v^{3}+64\left( 25\alpha +8\right) v^{2}+128\left( 11\alpha +4\right)
v+553\alpha +512}\right\} \\
&=&\frac{1+\dsum\limits_{n=1}^{m}\left( n+1\right) a_{n}z^{n}-\frac{\left(
8v+7\right) ^{3}\alpha }{256\left( 2+\alpha \right) v^{2}+8\left(
64+29\alpha \right) v+210\alpha +512}\dsum\limits_{n=m+1}^{\infty }\left(
n+1\right) a_{n}z^{n}}{1+\dsum\limits_{n=1}^{\infty }\left( n+1\right)
a_{n}z^{n}}
\end{eqnarray*}%
where 
\begin{equation*}
\left\vert p(z)\right\vert \leq \frac{\frac{512\alpha v^{3}+64\left(
25\alpha +8\right) v^{2}+128\left( 11\alpha +4\right) v+553\alpha +512}{%
256\left( 2+\alpha \right) v^{2}+8\left( 64+29\alpha \right) v+210\alpha +512%
}\dsum\limits_{n=m+1}^{\infty }\left( n+1\right) \left\vert a_{n}\right\vert 
}{2-2\dsum\limits_{n=1}^{m}\left( n+1\right) \left\vert a_{n}\right\vert -%
\frac{512\alpha v^{3}+64\left( 17\alpha -8\right) v^{2}+16\left( 59\alpha
-32\right) +133\alpha -512}{256\left( 2+\alpha \right) v^{2}+8\left(
64+29\alpha \right) v+210\alpha +512}\dsum\limits_{n=m+1}^{\infty }\left(
n+1\right) \left\vert a_{n}\right\vert }\leq 1.
\end{equation*}%
The last inequality is equivalent to 
\begin{equation}
\dsum\limits_{n=1}^{m}\left( n+1\right) \left\vert a_{n}\right\vert +\frac{%
\left( 8v+7\right) ^{3}\alpha }{256\left( 2+\alpha \right) v^{2}+8\left(
64+29\alpha \right) v+210\alpha +512}\dsum\limits_{n=m+1}^{\infty }\left(
n+1\right) \left\vert a_{n}\right\vert \leq 1.  \tag{2.10}
\end{equation}%
Since the left hand side of (2.10) is bounded above by $\frac{\left(
8v+7\right) ^{3}\alpha }{256\left( 2+\alpha \right) v^{2}+8\left(
64+29\alpha \right) v+210\alpha +512}\dsum\limits_{n=1}^{\infty }\left(
n+1\right) \left\vert a_{n}\right\vert ,$ the proof is complated.
\end{proof}

For some special cases of $v$ it is known that(see \cite{3})%
\begin{equation}
J_{\frac{1}{2}}(z)=\sqrt{\frac{2}{\pi z}}\sin z  \tag{2.11}
\end{equation}%
and%
\begin{equation}
J_{\frac{3}{2}}(z)=\sqrt{\frac{2}{\pi z}}\left( \frac{\sin z}{z}-\cos
z\right) .  \tag{2.12}
\end{equation}%
Now by using some special cases of $\alpha $ and $v$ in Theorem 2.1 and
Theorem 2.2, we can obtain following corollaries.

\begin{corollary}
If we take $\alpha =1,$ $v=\frac{1}{2}$ and $m=0,$ we have $w_{1,\frac{1}{2}%
}(z)=\frac{\sqrt{z}}{2}\left\{ \sin \sqrt{z}+\cos \sqrt{z}-\frac{\sin \sqrt{z%
}}{4}\right\} ,$ $w_{1,\frac{1}{2}}^{^{\prime }}(z)=\frac{1}{4\sqrt{z}}%
\left\{ \left( 1-\sqrt{z}\right) \sin \sqrt{z}+\frac{\left( 1+2\sqrt{z}%
\right) \cos \sqrt{z}}{2}\right\} ,\left( w_{1,\frac{1}{2}}\right) _{0}(z)=z$
and $\left( w_{1,\frac{1}{2}}\right) _{0}^{^{\prime }}(z)=1$. In view of
Theorem 2.2 we get the following inequalities:%
\begin{equation*}
\mathcal{R}\left\{ \frac{1}{4\sqrt{z}}\left( \left( 1-\sqrt{z}\right) \sin 
\sqrt{z}+\frac{\left( 1+2\sqrt{z}\right) \cos \sqrt{z}}{2}\right) \right\}
\geq \frac{45}{1286},
\end{equation*}%
\begin{equation*}
\mathcal{R}\left\{ \frac{8\sqrt{z}}{\left( 1+2\sqrt{z}\right) \cos \sqrt{z}%
-2\left( \sqrt{z}-1\right) \sin \sqrt{z}}\right\} \geq \frac{1331}{2617}.
\end{equation*}
\end{corollary}

\begin{corollary}
If we take $\alpha =\frac{3}{2},$ $v=\frac{1}{2}$ and $m=0,$ we have $w_{%
\frac{3}{2},\frac{1}{2}}(z)=\frac{\sqrt{z}}{3}(2\sin \sqrt{z}+\cos \sqrt{z})-%
\frac{\sin \sqrt{z}}{6},$ $w_{\frac{3}{2},\frac{1}{2}}^{^{\prime }}(z)=\frac{%
4\sin \sqrt{z}+\cos \sqrt{z}}{12\sqrt{z}}+\frac{2\cos \sqrt{z}-\sin \sqrt{z}%
}{6},$ $\left( w_{\frac{3}{2},\frac{1}{2}}\right) _{0}(z)=z$ and $\left( w_{%
\frac{3}{2},\frac{1}{2}}\right) _{0}^{^{\prime }}(z)=1.$ In view of Theorem
2.2 we get the following inequalities: 
\begin{equation*}
\mathcal{R}\left\{ \frac{4\sin \sqrt{z}+\cos \sqrt{z}}{12\sqrt{z}}+\frac{%
2\cos \sqrt{z}-\sin \sqrt{z}}{6}\right\} \geq \frac{1031}{2962},
\end{equation*}%
\begin{equation*}
\mathcal{R}\left\{ \frac{12\sqrt{z}}{\left( 4\sqrt{z}+1\right) \cos \sqrt{z}%
-2\left( \sqrt{z}-2\right) \sin \sqrt{z}}\right\} \geq \frac{3993}{6955}.
\end{equation*}
\end{corollary}

\begin{corollary}
If we take $\alpha =5,$ $v=\frac{3}{2}$ and $m=0,$ we have $w_{5,\frac{3}{2}%
}(z)=\frac{21}{10}\left( \frac{\sin \sqrt{z}}{\sqrt{z}}-\cos \sqrt{z}\right)
+\frac{9}{20}\left( \frac{\cos \sqrt{z}}{\sqrt{z}}-\frac{\sin \sqrt{z}}{z}%
\right) +\frac{3}{10}\sin \sqrt{z},$ $w_{5,\frac{3}{2}}^{^{\prime }}(z)=%
\frac{21}{20}\left( \frac{\cos \sqrt{z}}{z}-\frac{\sin \sqrt{z}}{z\sqrt{z}}+%
\frac{\sin \sqrt{z}}{\sqrt{z}}\right) +\frac{9}{20}\left( \frac{\sin \sqrt{z}%
}{z^{2}}-\frac{\cos \sqrt{z}}{z\sqrt{z}}\right) -\frac{9}{40}\frac{\sin 
\sqrt{z}}{z}+\frac{3}{20}\frac{\cos \sqrt{z}}{\sqrt{z}},$ $\left( w_{5,\frac{%
3}{2}}\right) _{0}(z)=z$ and $\left( w_{5,\frac{3}{2}}\right) _{0}^{^{\prime
}}(z)=1.$ In view of Theorem 2.1 we get the following inequalities: 
\begin{equation*}
\mathcal{R}\left\{ \frac{14z\left( \sin \sqrt{z}-\sqrt{z}\cos \sqrt{z}%
\right) +\left( 2z-3\right) \sqrt{z}\sin \sqrt{z}+3z\cos \sqrt{z}}{z^{5/2}}%
\right\} \geq \frac{20}{57},
\end{equation*}%
\begin{equation*}
\mathcal{R}\left\{ \frac{z^{5/2}}{14z\left( \sin \sqrt{z}-\sqrt{z}\cos \sqrt{%
z}\right) +\left( 2z-3\right) \sqrt{z}\sin \sqrt{z}+3z\cos \sqrt{z}}\right\}
\geq \frac{57}{740}.
\end{equation*}
\end{corollary}

\section{Illustrative examples and image domains}

In this section, illustrative examples along with the geometrical
descriptions of the image domains of the unit disk by the ratio of
normalized Dini function to its sequence of partial sums or the the ratio of
its sequence of partial sums to the function which we considered in our
corollaries in section 2, have been given in following figures.

\begin{example}
The image domains of $f_{1}(z)=\frac{1}{4\sqrt{z}}\left( \left( 1-\sqrt{z}%
\right) \sin \sqrt{z}+\frac{\left( 1+2\sqrt{z}\right) \cos \sqrt{z}}{2}%
\right) $ and $f_{2}(z)=\frac{8\sqrt{z}}{\left( 1+2\sqrt{z}\right) \cos 
\sqrt{z}-2\left( \sqrt{z}-1\right) \sin \sqrt{z}}$ are shown in Figure 1,
while the image domains of $f_{3}(z)=\frac{4\sin \sqrt{z}+\cos \sqrt{z}}{12%
\sqrt{z}}+\frac{2\cos \sqrt{z}-\sin \sqrt{z}}{6}$ and $f_{4}(z)=\frac{12%
\sqrt{z}}{\left( 4\sqrt{z}+1\right) \cos \sqrt{z}-2\left( \sqrt{z}-2\right)
\sin \sqrt{z}}$ \ are shown in Figure 2. 
\begin{equation*}
\begin{array}{cc}
\FRAME{itbpF}{2.7838in}{2.7838in}{0in}{}{}{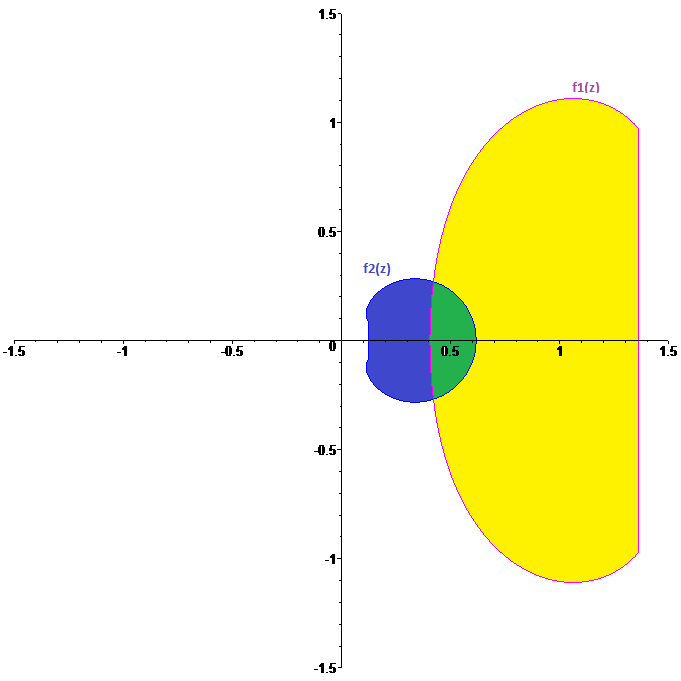}{\special{language
"Scientific Word";type "GRAPHIC";maintain-aspect-ratio TRUE;display
"USEDEF";valid_file "F";width 2.7838in;height 2.7838in;depth
0in;original-width 7.0941in;original-height 7.0941in;cropleft "0";croptop
"1";cropright "1";cropbottom "0";filename 'f1-f2.png';file-properties
"XNPEU";}} & \FRAME{itbpF}{2.7838in}{2.7838in}{0in}{}{}{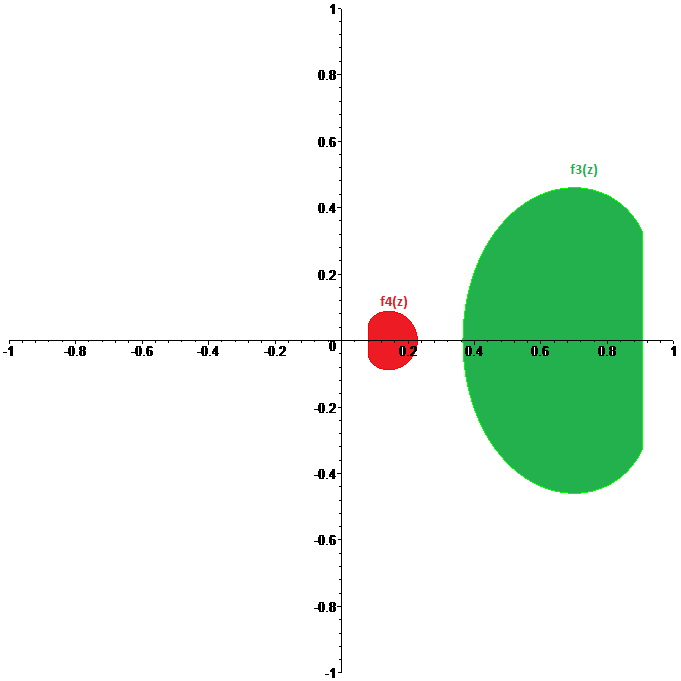}{\special%
{language "Scientific Word";type "GRAPHIC";maintain-aspect-ratio
TRUE;display "USEDEF";valid_file "F";width 2.7838in;height 2.7838in;depth
0in;original-width 7.0941in;original-height 7.0941in;cropleft "0";croptop
"1";cropright "1";cropbottom "0";filename 'f3-f4.png';file-properties
"XNPEU";}} \\ 
Figure\text{ }1 & Figure\text{ }2%
\end{array}%
\end{equation*}
\end{example}

\begin{example}
We have the image domains of $f_{5}(z)=\frac{14z\left( \sin \sqrt{z}-\sqrt{z}%
\cos \sqrt{z}\right) +\left( 2z-3\right) \sqrt{z}\sin \sqrt{z}+3z\cos \sqrt{z%
}}{z^{5/2}}$ and $f_{6}(z)=\frac{z^{5/2}}{14z\left( \sin \sqrt{z}-\sqrt{z}%
\cos \sqrt{z}\right) +\left( 2z-3\right) \sqrt{z}\sin \sqrt{z}+3z\cos \sqrt{z%
}}$ in Figure 3 and Figure 4, respectively.%
\begin{equation*}
\begin{array}{cc}
\FRAME{itbpF}{2.7838in}{2.7838in}{0in}{}{}{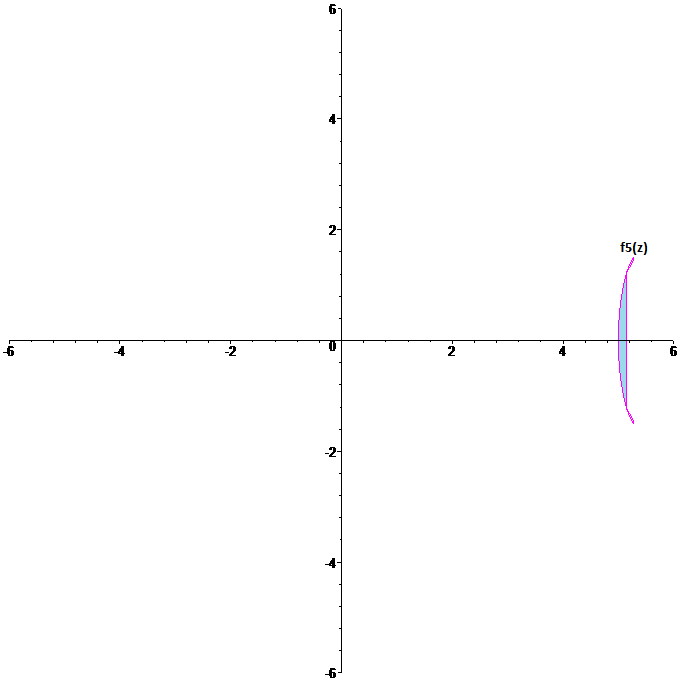}{\special{language
"Scientific Word";type "GRAPHIC";maintain-aspect-ratio TRUE;display
"USEDEF";valid_file "F";width 2.7838in;height 2.7838in;depth
0in;original-width 7.0941in;original-height 7.0941in;cropleft "0";croptop
"1";cropright "1";cropbottom "0";filename 'f5.png';file-properties "XNPEU";}}
& \FRAME{itbpF}{2.7838in}{2.7838in}{0in}{}{}{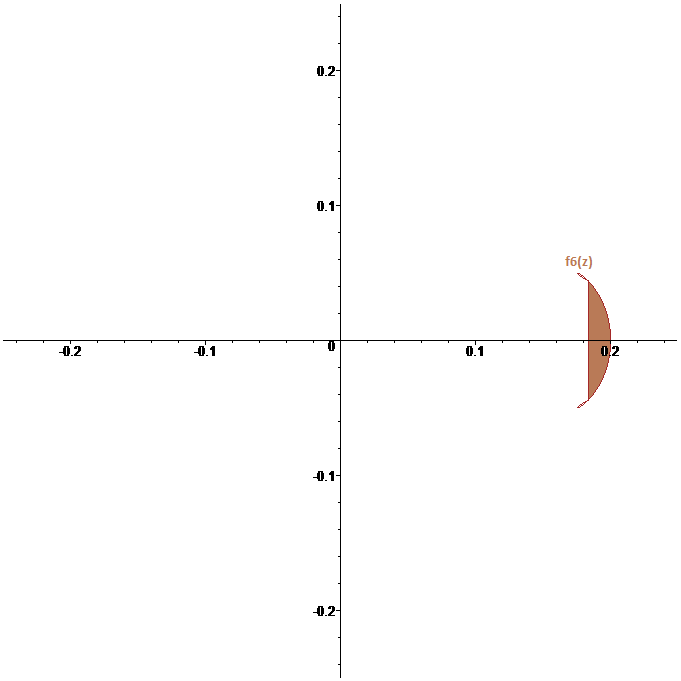}{\special{language
"Scientific Word";type "GRAPHIC";maintain-aspect-ratio TRUE;display
"USEDEF";valid_file "F";width 2.7838in;height 2.7838in;depth
0in;original-width 7.0941in;original-height 7.0941in;cropleft "0";croptop
"1";cropright "1";cropbottom "0";filename 'f6.png';file-properties "XNPEU";}}
\\ 
Figure\text{ }3 & Figure\text{ }4%
\end{array}%
\end{equation*}
\end{example}

\end{document}